\documentclass[twoside,tbtags,leqno]{amsart}
\usepackage{amsmath,color,amsthm,amssymb,cite}

\numberwithin{equation}{section}

\newtheorem{theorem}{Theorem}[section]
\newtheorem{corollary}[theorem]{Corollary}
\newtheorem{lemma}[theorem]{Lemma}
\newtheorem{proposition}[theorem]{Proposition}
\newtheorem{claim}[theorem]{Claim}
\newtheorem{example}[theorem]{\sl Example}
\newtheorem{definition}[theorem]{Definition}

\theoremstyle{definition}
\newtheorem{remark}[theorem]{Remark}

\newcommand{\EE}{{\bf  E}}

\newcommand{\PP}{{\bf  P}}

\newcommand{\qq}{{\bf  q}}

\newcommand{\Var}{{\bf Var}}

\newcommand{\pit}{\pi_{-0}}

\newcommand{\Pt}{{\widetilde{P}}}

\newcommand{\cL}{\mathcal{L}}

\newcommand{\Lc}{{\mathcal L}}
\newcommand{\Lto}{{\stackrel{\Lc}{\to}}}
\newcommand{\Leq}{{\,\stackrel{\Lc}{=}\,}}

\newcommand{\Xc}{{\mathcal X}}

\newcommand{\xh}{\hat{x}}
\newcommand{\zh}{\hat{0}}

\newcommand{\Ph}{\widehat{P}}

\newcommand{\Xh}{\widehat{X}}
\newcommand{\Th}{\widehat{T}}

\newcommand{\pih}{\hat{\pi}}

\newcommand{\begp}{\begin{proposition}}
\newcommand{\enp}{\end{proposition}}
\newcommand{\begt}{\begin{theorem}}
\newcommand{\ent}{\end{theorem}}
\newcommand{\begl}{\begin{lemma}}
\newcommand{\enl}{\end{lemma}}
\newcommand{\begc}{\begin{corollary}}
\newcommand{\enc}{\end{corollary}}
\newcommand{\begcl}{\begin{claim}}
\newcommand{\encl}{\end{claim}}
\newcommand{\begr}{\begin{remark}}
\newcommand{\enr}{\end{remark}}
\newcommand{\begal}{\begin{algorithm}}
\newcommand{\enal}{\end{algorithm}}
\newcommand{\begd}{\begin{definition}}
\newcommand{\enf}{\end{definition}}
\newcommand{\begx}{\begin{example}}
\newcommand{\enx}{\end{example}}
\newcommand{\bega}{\begin{array}}
\newcommand{\ena}{\end{array}}

\newcommand{\ignore}[1]{}

\def\rompar(#1){\textup(#1\textup)}    

\newcommand\gd{\delta}
\newcommand\gl{\lambda}
\newcommand\gL{\Lambda}

\newcommand{\refS}[1]{Section~\ref{#1}}
\newcommand{\refT}[1]{Theorem~\ref{#1}}
\newcommand{\refC}[1]{Corollary~\ref{#1}}
\newcommand{\refL}[1]{Lemma~\ref{#1}}
\newcommand{\refR}[1]{Remark~\ref{#1}}

\newcommand\ie{i.e.\spacefactor=1000}
\newcommand\eg{e.g.\spacefactor=1000}

\newcommand\tr{\operatorname{tr}}

\newcommand\urladdrx[1]{{\urladdr{\def~{{\tiny$\sim$}}#1}}}

\begin{document}

\newcommand{\tab}[0]{\hspace{.1in}}


\title[Hitting Times and Interlacing Eigenvalues]
{Hitting Times and Interlacing Eigenvalues:\\
A Stochastic Approach Using Intertwinings}

\author{James Allen Fill}
\address{Department of Applied Mathematics and Statistics,
The Johns Hopkins University,
34th and 
Charles Streets,
Baltimore, MD 21218-2682 USA}
\email{jimfill@jhu.edu}
\urladdrx{http://www.ams.jhu.edu/~fill/}
\thanks{Research supported by the Acheson~J.~Duncan Fund for the Advancement of Research in
Statistics.  
}
\author{Vince Lyzinski}
\address{Department of Applied Mathematics and Statistics,
The Johns Hopkins University,
34th and 
Charles Streets,
Baltimore, MD 21218-2682 USA}
\email{vincelyzinski@gmail.com}
\thanks{Research supported by the Acheson~J.~Duncan Fund for the Advancement of 
Research in Statistics, and by U.S. Department of Education GAANN grant P200A090128.}
\date{May~29, 2012.  Revised August~30, 2012.}

\begin{abstract} \small\baselineskip=9pt 
We develop a systematic matrix-analytic approach, based on intertwinings of Markov semigroups, for proving theorems about hitting-time distributions for finite-state Markov chains---an approach that (sometimes) deepens understanding of the theorems by providing corresponding sample-path-by-sample-path stochastic constructions.  We employ our approach to give new proofs and constructions for two theorems due to Mark Brown, theorems giving two quite different representations of hitting-time distributions for finite-state Markov chains started in stationarity.  The proof, and corresponding construction, for one of the two theorems elucidates an intriguing connection between hitting-time distributions and the interlacing eigenvalues theorem for bordered symmetric matrices.
\end{abstract}

\maketitle

\section{Introduction and Outline of our General Technique}
\label{S:intro}
Recently, stochastic proofs and constructions have been provided for some theorems that give explicit descriptions of Markov chain hitting-time distributions; previously known proofs of the theorems had been analytic in nature.  Specifically, Fill~\cite{JAFBD} and Diaconis and Miclo~\cite{PDLMBD} both give stochastic constructions for a famous birth-and-death hitting-time result first proven analytically by Karlin and McGregor~\cite{KM} in 1959.  
Fill~\cite{MR2530104} (see also Miclo~\cite{MR2654550}) extends to upward-skip-free and more general chains, in particular giving a (sometimes) stochastic proof for a hitting-time theorem for upward-skip-free chains established analytically by Brown and Shao~\cite{BrownShao87}. 

In Sections \ref{S:twine}-\ref{S:quasi} we describe a systematic approach, using intertwinings of Markov semigroups, for obtaining simple stochastic decompositions of the distributions of hitting times for Markov chains and also providing sample-path-by-sample-path constructions for the individual components in these decompositions.  
For example, if one can prove a theorem that the law of a certain Markov chain hitting time~$T$ is a convolution of Geometric distributions with certain parameters, our additional goal is to decompose~$T$ explicitly---sample path by sample path---as a sum of independent Geometric random variables with the specified parameters; this deepens understanding as to ``why'' the theorem is true.  See 
Fill~\cite{MR2530104} for a class of examples using this approach.  Our approach is essentially matrix-analytic, but if certain conditions elaborated in Sections \ref{S:twine}-\ref{S:strategy} are met, then our method also yields a decomposition for each sample path.  For the applications discussed in this paper, our approach provides new matrix-analytic proofs for hitting-time results which were previously only known via analytic methods (such as computation of Laplace transforms), and these new proofs provide new insights into the evolution of the Markov chain.  A simple example of our approach, with an application to the Moran model in population genetics, is presented in \refS{S:block}.  

We then employ our intertwinings approach to provide new proofs for two theorems due to Mark Brown, providing two quite different representations of hitting-time distributions for Markov chains started in stationarity.  The proof, and subsequent construction, for the first theorem (\refS{S:lace}) will elucidate an interesting connection between hitting-time distributions and the interlacing eigenvalues theorem for bordered symmetric matrices.  Application 
of our approach obtains a construction for the second theorem (\refS{S:V}) that results in a bonus: We are able to extend Brown's theorem from reversible chains to more general ones. 

\smallskip

\emph{Notation}:\ Throughout this paper, all vectors used are by default row vectors.  
We write $\delta_j$ for the vector of $0$'s except for a~$1$ in the $j$th position, and $\vec{1}$ for the vector of $1$'s.
The transpose of a matrix~$A$ is denoted by $A^T$.  The notation $A(:,j) := A \delta_j^T$ is used to denote the $j$th column of $A$, and $A(i,:) := \delta_i A$ to denote the $i$th row of~$A$.  For
any matrix~$A$, we let $A_0$ denote the principal submatrix of~$A$ obtained by deleting the topmost row and leftmost column.

\subsection{Intertwinings and sample-path linking}
\label{S:twine} 

The main conceptual tool in our approach is the notion of an intertwining 
of Markov semigroups, for which we now provide the needed background in the context (sufficient for our purposes) of finite-state Markov chains.  For further background on intertwinings, see \cite{BIA}, \cite{CARPET}, \cite{ROGPIT}.  Suppose that we have two state spaces, the first (``primary") of size $n$ and the second (``dual") of size $\hat{n}$.  Let~$P$ be the transition matrix of a
Markov chain~$X$, begun in distribution $\pi_0$, on the primary state space.  [We write $X \sim (\pi_0, P)$ as shorthand.]  Similarly, let $\widehat{P}$ be the transition matrix of a Markov chain~$\Xh$,  begun in 
$\pih_0$, on the dual state space.  Let $\Lambda$ be an $\hat{n}$-by-$n$ stochastic matrix.
  
\begin{definition}
\label{D:twine}
\emph{
We say that the Markov semigroups $(P^t)_{t=0,1,2,\ldots}$ and $(\widehat{P}^t)_{t=0,1,2,\ldots}$ are
\emph{intertwined by the 
link~$\Lambda$} (or, for short, that~$P$ and~$\Ph$ are intertwined by the link~$\gL$) if
$$\Lambda P = \widehat{P} \Lambda;$$
and we say that $(\pi_0, P)$ and $(\pih_0, \Ph)$ are \emph{intertwined by $\Lambda$} if additionally
$$\pi_0=\pih_0 \Lambda.$$
}
\end{definition}
\smallskip

Here are three consequences when $(\pi_0, P)$ and $(\pih_0, \widehat{P})$ are intertwined by $\Lambda$ (with the first two immediate---for example, $\Lambda P^2 = \Ph \Lambda P = \Ph^2 \Lambda$---and the third \emph{crucial} for our purposes):
\begin{itemize}
\item For $t = 0, 1, 2, \dots$, we have $\Lambda P^t = \Ph^t \Lambda$.
\item For $t = 0, 1, 2, \dots$, the distributions $\pi_t$ and $\pih_t$ at time~$t$ satisfy $\pi_t = \pih_t \Lambda$.
\item Given $X \sim (\pi_0, P)$, one can build $\widehat{X}_t$ from $X_0, \dots, X_t$ and randomness independent of~$X$ so that
$\widehat{X}\sim (\pih_0, \widehat{P})$ and the conditional law of $X_t$ given $(\Xh_0, \dots, \Xh_t)$ has probability mass function given by the $\Xh_t$-row of~$\gL$:
\begin{equation}
\label{linking}
\Lc(X_t\,|\,\Xh_0, \dots, \Xh_t) = \gL(\Xh_t, \cdot), \quad t = 0, 1, 2, \dots.
\end{equation}
\end{itemize}
We call this last consequence \emph{sample-path linking}, and will explain next, once and for all, (a)~how it is done and 
(b)~why it is useful for hitting-time (or mixing-time) constructions.  We will then have no need to repeat this discussion when we turn to applications, each of which will therefore culminate with the explicit construction of an intertwining (or at least of a 
quasi-intertwining, as discussed in \refS{S:quasi}).

Whenever we have an intertwining of $(\pi_0, P)$ and $(\pih_0, \Ph)$, Section~2.4 of the strong stationary duality paper~\cite{DFSST} by Diaconis and Fill gives a family of ways to create sample-path linking.  Here is one~\cite[eq.~(2.36)]{DFSST}, with
$\Delta := \Ph \gL = \gL P$:
\begin{itemize}
\item Set $\Xh_0 \leftarrow \hat{x}_0$ with probability $\pih_0(\hat{x}_0) \gL(\hat{x}_0, x_0) / \pi_0(x_0)$.
\item Inductively, for $t \geq 1$, set $\Xh_t \leftarrow \hat{x}_t$ 
with probability 
$$
\Ph(\hat{x}_{t - 1}, \hat{x}_t) \gL(\hat{x}_t, x_t) / \Delta(\hat{x}_{t - 1}, x_t).
$$
\end{itemize}

Suppose $(\pi_0, P)$ and $(\pih_0, \Ph)$ are intertwined and that, given $X \sim (\pi_0, P)$, we have created linked sample paths for
$\Xh \sim (\pih_0, \Ph)$,
as at~\eqref{linking}.
Suppose further that there are states, call them~$0$ and~$\zh$, such that~$0$ (respectively, $\zh$) is the unique absorbing state 
for~$P$ (resp.,\ $\Ph$) and that
\begin{equation}
\label{abs}
\gL \delta_0^T = \delta_{\zh}^T,
\end{equation}
\ie, that $\gL(\zh, 0) = 1$ and $\gL(\xh, 0) = 0$ for $\hat{x} \neq \zh$.
Then, for the bivariate process $(\Xh, X)$, we see that absorption times agree: $T_0(X) = T_{\zh}(\Xh)$.
For a parallel explanation of how sample-path linking can be used to connect the \emph{mixing} time for an ergodic primary chain with a hitting time for a dual chain, consult \cite{DFSST}; very closely related is the FMMR perfect sampling 
algorithm~\cite{MR99g:60113,MR2001m:60164}.

\subsection{Strategy for absorption-time decompositions}
\label{S:strategy}

The two hitting-time theorems discussed in Sections~\ref{S:lace}--\ref{S:V} both concern ergodic Markov chains.  However, since for these theorems we have no interest in the chain after the specified target state~$0$ has been hit, the hitting-time distribution for such a chain is the same as the absorption-time distribution for the corresponding chain for which the target state is converted to absorbing by replacing the row of $P$ corresponding to state~$0$ by the row vector $\delta_0$.  

It should also be noted that hitting-time theorems and stochastic constructions are easily extended to hitting times
of general subsets~$A$, by the standard trick of collapsing~$A$ to a single state. 

Here is then a general strategy for obtaining a decomposition of the time to absorption in state~$0$ of a Markov chain $X \sim (\pi_0, P)$ from a decomposition of its distribution:

\begin{enumerate}

\item[1.] Discover another chain $\Xh \sim (\pih_0, \Ph)$ for which the sample-point-wise decomposition of the time to absorption in state~$\zh$ is \emph{clearly} of the form specified for~$X$.  (For example, for a pure-death chain started at~$d$ with absorbing state $\zh = 0$, the time to absorption is clearly the sum of independent Geometric random variables.)  

\item[2.] Find a link~$\gL$ that intertwines $(\pi_0, P)$ and $(\pih_0, \Ph)$.

\item[3.] Prove the condition~\eqref{abs}.

\item[4.] Conclude from the preceding discussion that (after sample-path linking) $T_0(X) = T_{\zh}(\Xh)$ and use the 
sample-point-wise decomposition for $T_{\zh}(\Xh)$ as the decomposition for $T_0(X)$.

\end{enumerate}

\noindent An early use of our strategy (adapted for mixing times, rather than absorption times) was in connection with the theory of 
strong stationary duality~\cite{DFSST}, for which the fullest development has resulted in the case of set-valued strong stationary duality (see especially \cite[Secs.\ 3--4]{DFSST} and~\cite{MR99g:60113}; very closely related is the technique of \emph{evolving sets}~\cite{MR2198701}).  For a very recent application to hitting times and fastest strong stationary times for birth and death chains, see~\cite{JAFBD} and~\cite{MR2530104}.

\subsection{Quasi-intertwinings}
\label{S:quasi}

Suppose that the (algebraic) intertwining conditions
$\gL P = \Ph \gL$ and $\pi_0 = \pih_0 \gL$
hold for some \emph{not necessarily stochastic} matrix~$\gL$ with rows summing to unity.
We call this a \emph{quasi-intertwining} of $(\pi_0, P)$ and $(\pih_0, \Ph)$
by the \emph{quasi-link} $\gL$.  Then we again have the identities
$\gL P^t = \Ph^t \gL$ and $\pi_t = \pih_t \gL$.  As before, suppose further
that ~\eqref{abs} holds.
Then, although (if~$\gL$ is not stochastic) we cannot do sample-path linking and so cannot achieve $T_0(X) = T_{\zh}(\Xh)$,
we can still conclude that $T_0(X)$ and $T_{\zh}(\Xh)$ have the same distribution, because
$$
\mbox{$\PP(T_0(X) \leq t) = \pi_t(0) = \sum_{\xh} \pih_t(\xh) \gL(\xh, 0) = \pih_t(\zh) = \PP(T_{\zh}(\Xh) \leq t)$}.
$$

\begr
\label{R:prod}
The following easily-verified observations will be used in our application in \refS{S:lace}.

(a)~If $\gL_1$ is a quasi-link providing a quasi-intertwining of $(\pi_0, P)$ and $(\pi^*_0, P^*)$ and $\gL_2$ is similarly a quasi-link from $(\pi^*_0, P^*)$ to $(\pih_0, \Ph)$, then $\gL := \gL_2 \gL_1$ is a quasi-link from $(\pi_0, P)$ to 
$(\pih_0, \Ph)$.

(b)~If, additionally, the chains have respective unique absorbing states $0, 0^*, \zh$ and \eqref{abs} holds for~$\gL_1$ and 
for~$\gL_2$ (\ie, $\gL_1 \delta_0^T = \delta_{0^*}^T$ and $\gL_2 \delta_{0^*}^T = \delta_{\zh}^T$), then \eqref{abs} holds also for $\gL$ (\ie, $\gL \delta_0^T = \delta_{\zh}^T$).

(c)~If $\gL_1$ and~$\gL_2$ in~(a) are both links, then so is~$\gL$. 
\enr

\section{An illustrative example:\ Block chains and the Moran model}
\label{S:block}

\subsection{Block chains}
\label{SS:block}

In this section
we warm up to the main applications of Sections \ref{S:lace}--\ref{S:V} by providing a simple application of the technique outlined in \refS{S:intro}.  Let~$P$ be a Markov kernel on finite state space~$\Xc$ with the following block structure: 
\begin{equation}
\label{block}
P = 
\left( 
\begin{array}{cccccc}
P_{00} & P_{01} & P_{02} & \hdots & P_{0k} \\
P_{10} & P_{11} & P_{12} & \hdots & P_{1k} \\
P_{20} & P_{21} & P_{22} & \hdots & P_{2k} \\
\vdots & \vdots &\vdots  & \ddots  & \vdots  \\
P_{k0} & P_{k1} & P_{k2} & \hdots & P_{kk}
\end{array}
\right).
\end{equation}
For $i = 0, \dots, k$, let $\mu_i$ be a Perron left eigenvector of $P_{ii}$ [that is, a nonzero row vector with nonnegative entries such that
$$
\mu_i P_{ii} = \rho(P_{ii}) \mu_i,
$$
where $\rho(A)$ denotes the spectral radius of a matrix~$A$], normalized to sum to~$1$.
It is well known (\eg,\ \cite[Theorem~8.3.1]{HJ}) that such an eigenvector exists; when, additionally, $P_{ii}$ is irreducible, the vector $\mu_i$ is unique (\eg,\ \cite[Theorem~8.4.4]{HJ}) and is often called the quasi-stationary distribution for $P_{ii}$.  We make the following special assumption concerning~$P$:  For every~$i$ and~$j$, the vector $\mu_i P_{ij}$ is proportional to $\mu_j$, say $\mu_i P_{ij} = \Ph(i, j) \mu_j$.  In words, the chain with transition matrix~$P$, started in distribution $\mu_i$ over block~$i$, moves in one step to block~$j$ with probability $\Ph(i, j)$; and, conditionally given that it moves to block~$j$, it ``lands'' in block~$j$ with distribution $\mu_j$.  We note in passing that~$\Ph$ is a ($k + 1$)-by-($k + 1$) matrix, and that $\Ph(i, i) = \rho(P_{ii})$ for every~$i$.  Define a 
($k + 1$)-by-$|\Xc|$ stochastic matrix~$\gL$ by setting
\begin{equation}
\label{blocklink}
\gL := 
\left( 
\begin{array}{cccccc}
\mu_0 & 0 & 0 & \hdots & 0 \\
0 &\mu_1 & 0  &\hdots & 0 \\
0 & 0 & \mu_2  & \hdots & 0 \\
\vdots & \vdots &\vdots  & \ddots  & \vdots  \\
0 & 0&0 & \hdots & \mu_k
\end{array}
\right).
\end{equation}

Now consider a chain~$X$ with transition matrix~$P$ and initial distribution $\pi_0$; suppose that $\pi_0$ is a mixture, say $\sum_{i = 0}^k \pih_0(i) \mu_i$, of the distributions $\mu_i$ (each of which can be regarded naturally as a distribution on the entire state space).

\begp
\label{P:block}
In the block-chain setting described above, $(\pi_0, P)$ and $(\pih_0, \Ph)$ are intertwined by the link~$\gL$.
\enp

\begin{proof}
The proof is a simple matter of checking Definition~\ref{D:twine} by checking that the identity 
$\mu_i P_{ij} \equiv \Ph(i, j) \mu_j$ gives $\gL P = \Ph \gL$ and that the assumption $\pi_0 = \sum_{i = 0}^k \pih_0(i) \mu_i$ gives $\pi_0 = \pih_0 \gL$.
\end{proof}

The sample-path linking developed in \refS{S:twine} is very simple to describe in our present block-chain setting: 
$\Xh_t$ is simply the block ($\in \{0, \dots, k\}$) to which $X_t$ belongs.  This simple description is due to the very simple nature of the link~\eqref{blocklink}; the sample-path linking is more complicated for the applications in Sections \ref{S:lace}--\ref{S:V}.

\subsection{The Moran model}
\label{SS:Moran}

We now apply the block-chain development in the preceding subsection to a Markov chain on partitions of the 
positive integer $n$ introduced in~\cite{mccollab} as a somewhat light-hearted model for collaboration among mathematicians.  Their model is precisely the Moran model from population genetics according to the following 
definition \cite[Definition 2.26]{MR2759587} modified (a)~to switch in natural fashion from continuous time to discrete time and (b)~to limit the description of the state at each unit of time by distinguishing between genes with different labels but otherwise ignoring the values of the labels:
\begin{quote}
A population of~$N$ genes evolves according to the Moran model if at exponential rate ${N \choose 2}$ a pair of genes is sampled uniformly at random from the population, one dies and the other splits in two.
\end{quote}
The chain we will consider here is a simple example of a coalescent chain, a class popularized in the seminal works of Kingman (see for example \cite{king}, \cite{king2}, \cite{king3}).  For a more complete modern picture of the application and study of coalescing chains, see \cite{durr}.

Let~$S$ be a set of~$n$ indistinguishable objects.  (The objects are gene labels in the Moran model and are mathematicians in~\cite{mccollab}.)  The 
Markov chain of interest in~\cite{mccollab} 
is 
more easily described if we make use of the natural bijection between partitions of the integer~$n$ and set partitions of~$S$ obtained by identifying a partition $(n_1, n_2, \ldots, n_r)$ (with $1 \leq r < \infty$ and $n_1 \geq n_2 \geq \cdots \geq n_r \geq 1$) of the integer~$n$ with a partition of~$S$ into~$r$ indistinguishable subsets where the subsets are of sizes $n_1, n_2, \dots, n_r$.  Accordingly, if the present state of the Markov chain is the partition 
$(n_1, n_2, \ldots, n_r)$, then, viewing this as a partition of~$S$, uniformly select an ordered pair of unequal objects from~$S$, and suppose that the first and second objects are currently in subsets of size $n_i$ and $n_j$, respectively.   The transition is realized by moving the second object from the second subset to the first, resulting in two new subsets of sizes $n_i + 1$ and $n_j - 1$.  For example, if $n=6$ and the Markov chain is currently in the partition $(4,1,1)$, then with probability $8/30$ the chain transitions to $(5,1)$; with probability $2/30$, to $(4,2)$; with probability $8/30$, to $(3,2,1)$; and with probability $12/30$ the chain stays in $(4,1,1)$.  The authors of~\cite{mccollab} are concerned with the distribution of the hitting time of state $(n)$, the (absorbing) single-part partition, when the chain is begun in the 
$n$-parts partition $(1, \dots, 1)$.

Collecting partitions into blocks, where block~$i$ contains all partitions with~$i$ parts ($1 \leq i \leq n$), it is clear that the transition matrix~$P$ for this chain is block upper bidiagonal, since a one-step transition can only change the number of parts by~$0$ or $-1$. 
For example, in the simple case $n=4$, one possible ordering of the partitions by decreasing number of parts is $(1,1,1,1)$, $(2,1,1)$, $(2,2)$, $(3,1)$, $(4)$ and the corresponding $P$ is given by 
$$
P= \left( \begin{array}{cccc}
P_{44} & P_{43} & 0 & 0 \\
0 & P_{33} & P_{32} & 0 \\
0 & 0 & P_{22} & P_{21} \\
0 & 0 & 0  & P_{11} \end{array} \right)
=\left( \begin{array}{c|c|cc|c}
0 & 1 & 0 & 0 & 0\\
\hline 0 & 6/12 & 2/12 & 4/12 & 0\\
\hline 0 & 0 & 4/12 & 8/12 & 0\\
0 & 0 & 3/12 & 6/12 & 3/12\\
\hline 0 & 0 & 0 & 0 & 1\end{array} \right).
$$

We will make use of results in~\cite{mccollab} to see that~$P$ satisfies the assumptions of \refS{SS:block}.  To describe the results, let $1 \leq t \leq n$ and consider a partition~${\bf r}$ of~$n$ with~$t$ parts.  For $i=1,\ldots,n$, let $r_i$ be the number of parts of $\bf{r}$ equal to~$i$, so that $\sum_i i r_i = n$.  Let $m_{{\bf r}} := {t \choose {r_1, r_2, \dots, r_n}}$.  Define $\mu_t$ to be the row vector, supported on partitions of size $t$, whose entry corresponding to partition~${\bf r}$ is ${{n - 1} \choose {t - 1}}^{-1} m_{{\bf r}}$.  For $1 \leq t \leq n$, define $\gl_t := 1-\frac{t (t - 1)}{n (n - 1)}$.  
For example, if $n = 4$ and $t = 2$ and partitions with $2$ parts are listed (as above) in the order $(2, 2)$, $(3, 1)$, then $\mu_2 = (1 / 3, 2/ 3)$ and $\gl_2 =  5 / 6$.  Let the dual state space be ordered $n, n - 1, \dots, 1$ (corresponding naturally to the ordering we have used for the primary state space).  Define~$\gL$ by \eqref{blocklink}, but with the nonzero blocks correspondingly in \emph{decreasing} order $\mu_n, \mu_{n - 1}, \dots, \mu_1$ of subscript.  Let
$$
\Ph := \left( \begin{array}{ccccc}
\lambda_n & 1-\lambda_n & 0 & \cdots & 0\\
0 & \lambda_{n-1} & 1-\lambda_{n-1} & \cdots & 0 \\
\vdots & \vdots & \vdots & \ddots & \vdots \\
0 &  0 & 0 & \cdots  & \lambda_1 \end{array} \right).
$$
From Theorems~2 and~4 of~\cite{mccollab} we can use our Proposition~\ref{P:block} to derive easily the following intertwining result.

\begp
\label{P:partitions_intertwining}
Let $\pi_0$ be unit mass at the partition $(1, \dots, 1)$.  Then $(\pi_0,P)$ and $(\delta_n, \Ph)$ are intertwined by the link~$\gL$.
\enp

As a direct consequence of Proposition~\ref{P:partitions_intertwining}, we get the following hitting-time result.
\begc
\label{C:partitions} 
For fixed~$n$, the law of the time to absorption in state $(n)$ for the partitions-chain started in $(1, \dots, 1)$ is that of 
$\sum_{t=2}^n Y_{n, t}$ where $Y_{n, t} \sim \mathrm{Geo}(1 - \gl_{n, t})$, with $\gl_{n, t} = 1-\frac{t (t - 1)}{n (n - 1)}$, are independent.
\enc

In \cite{mccollab}, the authors were able to identify a simple expression for the expected hitting time of state $(n)$ when the chain is started in $\pi_0 = \delta_{(1, \dots, 1)}$, and challenged the reader to discover a pattern for the associated variance.  The authors found that $\EE_{\pi_0}\,T_{(n)}=(n-1)^2$.  This is confirmed by our \refC{C:partitions}, as 
$$
\EE_{\pi_0}\,T_{(n)} = \EE\,\sum_{k=2}^n Y_{n, k} = \sum_{k=2}^n \frac{n(n-1)}{k(k-1)}=(n-1)^2.
$$ 
Similarly, letting $H^{(2)}_n := \sum_{j = 1}^n j^{-2}$ denote the $n$th second-order harmonic number, we find
\begin{align*}
\Var_{\pi_0}\,T_{(n)} 
 &= \Var\,\sum_{k=2}^n Y_{n, k} = \sum_{k=2}^n \left(\bigg[\frac{n(n-1)}{k(k-1)}\bigg]^2-\frac{n(n-1)}{k(k-1)}\right) \\
 &= 2 [n (n - 1)]^2 H^{(2)}_n - (n - 1)^2 (3 n ^2 - 2 n + 2) \\
 &\sim (\mbox{$\frac{\pi^2}{3}$} - 3) n^4\mbox{\ \ as $n \to \infty$}.
\end{align*}

Proceeding further, it is not difficult to show that, when the partition chain is started in $\pi_0$, we have
$$
\frac{T_{(n)}}{n^2}\,\Lto\,S_{\infty} := \sum_{j = 2}^{\infty} X_j
$$
for independent random variables
$$
X_j \sim \mbox{Exp}(j (j - 1)), \quad j = 2, 3, \dots,
$$
with convergence of moments of all orders and (pointwise) of moment generating functions.  We omit the details.

\section{Hitting times and interlacing eigenvalues}
\label{S:lace}


\subsection{Brown's theorem}
\label{S:Brown1}

Our next construction will provide insight into a hitting-time result of Mark Brown \cite{MB} that elegantly connects the hitting time of a state for a reversible Markov chain started in stationarity to the celebrated interlacing eigenvalues theorem of linear algebra 
(see, \eg, Theorem~4.3.8 in \cite{HJ}).  We now proceed to set up Brown's result. 

Let $(X_t)_{t=0,1,2,\ldots}$ be a time-reversible ergodic discrete-time Markov chain with transition matrix $P$ on finite state space 
$\Xc=\lbrace 0,1, \ldots,n\rbrace$ with stationary distribution $\pi$.  If we let $D := \mbox{diag}(\pi(0),\ldots,\pi(n))$, then reversibility of $P$ implies that $S := D^{1/2} P D^{-1/2}$ is a symmetric matrix and thus $P$ has a real spectrum and a basis of real eigenvectors.  Denote the eigenvalues of $P$ 
by $1=\theta_0>\theta_1\geq\cdots\geq\theta_n>-1$.

Recall that, for any matrix~$A$, the principal submatrix of~$A$ obtained by deleting row~$0$ and column~$0$ is denoted $A_0$.  Denote the eigenvalues of $P_0$ by $\eta_1\geq \cdots \geq \eta_n$.  
Note that $S_0 = D_0^{1/2} P_0 D_0^{-1/2}$ is symmetric; by the interlacing eigenvalues theorem for bordered symmetric matrices 
(\eg, \cite[Theorem 4.3.8]{HJ}), the eigenvalues of~$P$ and $P_0$ 
interlace: $\theta_0>\eta_1\geq\theta_1\geq\cdots\geq\eta_n\geq\theta_n$.  
Cancel out common pairs of eigenvalues from the spectra $\sigma(P)$ and 
$\sigma(P_0)$ as follows. 
Consider $\sigma(P)$ and $\sigma(P_0)$ as multisets and remove the multiset $\sigma(P)\cap\sigma(P_0)$ from each of 
$\sigma(P)$ and $\sigma(P_0)$.  Relabel the reduced set of eigenvalues of $P$ as $\lbrace\lambda_i\rbrace_{i=0}^{r}$ with 
$\lambda_0\geq \lambda_1\geq\cdots\geq\lambda_r$ and of $P_0$ as $\lbrace\gamma_i\rbrace_{i=1}^{r}$ with $\gamma_1\geq\cdots\geq\gamma_r$.  After this cancellation, it is clear that the remaining eigenvalues strictly  
interlace: $1=\lambda_0>\gamma_1>\lambda_1>\cdots>\gamma_r>\lambda_r > -1$.

In what follows we need to assume 
that $\lambda_r \geq 0$. This is a rather harmless assumption, since we can if necessary shift attention from~$P$ to $\frac{1}{1+c}(P+cI)$ for suitably large~$c$.
  
Brown found it convenient to work in continuous time, but he could just as easily have proven the analogous result in our present discrete-time setting.  To state Brown's original continuous-time result, we make use of a very standard technique to produce a continuous-time chain from a discrete-time chain, by using independent and identically distributed (iid) Exp(1) holding times (in place of unit times) between transitions.  This continuous-time chain is sometimes called the \textit{continuization} of the Markov chain with one-step transition matrix $P$, and it has generator matrix $Q=P-I$.

Brown's original result can be stated as follows.

\begin{theorem}  
\label{T:Brown1c}
Let $Q=P-I$ be the generator of the continuization of a Markov chain with one-step transition matrix~$P$.  In the continuized chain, the distribution (or law) $\cL_{\pi} T_0$ of the hitting time of state~$0$ when the chain is started in stationarity, is that of $\sum_{i=1}^r Y_i$, where $Y_1, Y_2,\ldots,Y_r$ are independent and the distribution of $Y_i$ is the ``modified Exponential'' mixture
$$
Y_i \sim  \frac{1-\gamma_i}{1-\lambda_i} \delta_0 + \left( 1- \frac{1-\gamma_i}{1-\lambda_i} \right) \mathrm{Exp}(1 - \gamma_i)
$$
of unit mass at~$0$ and the Exponential distribution with parameter $1 - \gamma_i$;
the $\lambda$'s and $\gamma$'s are defined as above. 
\end{theorem}
We find it more convenient to work in discrete time, where the corresponding theorem (involving Geometric, rather than Exponential, distributions) is as follows.
\begin{theorem}
\label{T:Brown1d}  
In the discrete-time setting outlined above, $\cL_{\pi} T_0$ is the distribution of $\sum_{i=1}^r Y_i$, where 
$Y_1, Y_2, \ldots,Y_r$ are independent with the following ``modified Geometric'' distributions:
\begin{equation}
\label{yidist}
Y_i \sim  \frac{1-\gamma_i}{1-\lambda_i} \delta_0 + \left( 1- \frac{1-\gamma_i}{1-\lambda_i} \right) \mathrm{Geo}(1 - \gamma_i).
\end{equation}
\end{theorem}

We have our choice of working in discrete or continuous time because, fortunately, for any finite-state Markov chain and any target state~$0$ there is a simple relationship between hitting-time distributions in the two cases.  Let $T_0^{\mathrm{d}}$ be the time to hit state~$0$ in the discrete-time chain 
$(X_t)_{t=0,1,2,\ldots}$ with transition matrix $P$, and let $T_0^{\mathrm{c}}$ be the corresponding hitting time in the continuized chain.  Then the Laplace transform 
$\psi_{T_0^{\mathrm{c}}}(s):=\mathbf{E}\,\exp(-s\,T_0^{\mathrm{c}})$ and the probability generating function 
$G_{T_0^{\mathrm{d}}}(z):=\mathbf{E}\,z^{T_0^{\mathrm{d}}}$ of the hitting times satisfy a simple relationship:

\begl
\label{L:continuous-discrete}
For any finite-state discrete-time Markov chain and any target state~$0$, we have the following identity relating the distributions of the hitting time of state~$0$ for the continued chain and the discrete-time chain: 
$$
\psi_{T_0^\mathrm{c}}(s)=G_{T_0^\mathrm{d}} \left( \frac{1}{1+s} \right), \ s\geq 0.
$$
\enl
\begin{proof}
Let $X_i\sim \text{Exp}(1)$ be iid and independent of $T_0^{\mathrm{d}}$.  By definition of the continuized chain, we have $T_0^{\mathrm{c}}\Leq \sum_{i=1}^{T_0^{\mathrm{d}}} X_i$.  Then
$$
\psi_{T_0^\mathrm{c}}(s) = \mathbf{E}\ \mathrm{exp}(-s T_0^{\mathrm{c}})
  = \mathbf{E}\ \mathrm{exp}\left(-s \sum_{i=1}^{T_0^{\mathrm{d}}} X_i\right)
  = \mathbf{E}\left(\frac{1}{1+s}\right)^{T_0^{\mathrm{d}}}
  = G_{T_0^{\mathrm{d}}}\left(\frac{1}{1+s}\right).
$$
\end{proof}
This lemma allows us to easily derive \refT{T:Brown1c} from \refT{T:Brown1d} (and vice versa), since for $s \geq 0$ we have
\begin{align*}
\psi_{T_0^\mathrm{c}}(s) = G_{T_0^{\mathrm{d}}}\left(\frac{1}{1+s}\right) 
  &= \prod_i \left[ \frac{1 - \gamma_i}{1 - \lambda_i} + \left( 1 - \frac{1 - \gamma_i}{1 - \lambda_i} \right)
\frac{\frac{1 - \gamma_i}{1 + s}}{1 - \frac{\gamma_i}{1 + s}} \right] \\
  &= \prod_i \left[ \frac{1 - \gamma_i}{1 - \lambda_i} + \left( 1 - \frac{1 - \gamma_i}{1 - \lambda_i} \right) 
\frac{1 - \gamma_i}{1 -\gamma_i + s} \right].
\end{align*}

Our 
main result of \refS{S:lace} is another proof for~\refT{T:Brown1d}, culminating in our \refT{T:biglink} (see also the last paragraph of \refS{S:biglink}). 
Our proof provides---at least when the quasi-link $\Lambda$ we construct is a \emph{bona fide} link---an explicit stochastic construction of the hitting time of state~$0$ from a stationary start as a sum of independent modified Geometric random variables.  We tackle our proof of \refT{T:Brown1d} in two stages:\ in \refS{S:star} we build a certain ``star chain'' (random walk on a  weighted star graph) from the given chain and prove \refT{T:Brown1d} when this star chain is substituted for the given chain, and in \refS{S:starlink} we attempt to ``link" the given chain with the star chain of \refS{S:star}.  In \refS{S:biglink} we combine the results of 
Sections~\ref{S:star}--\ref{S:starlink} and provide our complete proof of \refT{T:Brown1d}.  
We could equally well prove the continuous-time analogues of all of our theorems and then apply the analogous intertwining results outlined in 
Section~2.3 of~\cite{SSDC} to provide (again when~$\Lambda$ is a link) an explicit continuous-time stochastic construction for \refT{T:Brown1c}.  We choose to work in discrete time for convenience and because, we believe, the ideas behind our constructions are easier to grasp in discrete time.

\subsection{A stochastic construction for the star chain}
\label{S:star}

Carrying out step~1 of the four-step strategy outlined in \refS{S:strategy} (finding a chain~$\Xh$ for which the hitting time of state~$\zh$ can be decomposed as a sum of independent modified Geometric random variables) turns out not to be too difficult; this step is carried out later, in \refL{L:hat}.  However, step~2 (finding a link~$\Lambda$ between the given~$X$ and~$\Xh$) proved challenging to us, so we break it down into two substeps, as described at the end of the preceding subsection.  In this subsection we build an ergodic star chain $X^*$ from the given chain~$X$ and show that the Markov semigroups for~$X^*$ (with the target state~$0^*$ converted to absorbing) and~$\Xh$ are intertwined by a link $\Lambda_2$.  
The state spaces for $X^*$ and $\Xh$ will both be $\{0, \dots, r\}$, and the roles of $\zh$ and $0^*$ will both be played by state~$0$.
For the star chain, we make full use of the notation in \refS{S:Brown1}.  The ``star'' has ``hub'' at~$0$ and ``spokes'' terminating at vertices $1, \dots, r$.  
The $r$-spoke star chain we build has previously been constructed in~\cite{AB}.

For the sake of brevity it is convenient to establish some additional notation.  Define 
$$
\rho_i := \frac{1-\gamma_i}{1-\gl_i}\quad\text{for\ }i = 1, \dots, r,
$$
and for $0\leq k\leq r$ define
\begin{equation}
\label{p}
\pi^*_k(i) :=
\begin{cases}
(1 - \rho_i) \prod_{1\leq j\leq k,\ j \neq i} \frac{1-\gamma_j - \rho_j (1-\gamma_i)}{\gamma_i - \gamma_j} & \mathrm{for\ }i = 1, \dots, k \\
\prod_{j = 1}^k \rho_j & \mathrm{for\ }i = 0.
\end{cases}
\end{equation}  
Set $\pi^*_k : = (\pi^*_k(0), \dots, \pi^*_k(k), 0, \dots, 0) \in \mathbb{R}^{r + 1}$
and note that $\pi^*_0 = \delta_0$.
The following lemma lays out the ergodic star chain of interest corresponding to the given chain.
\begl
\label{L:p}
\ \smallskip \\
{\rm (a)}~For all $0\leq k\leq r$ we have $\pi^*_k(i) > 0$ for $i = 0, \dots, k$ and $\sum_{i = 0}^k \pi^*_k(i) = 1$.\\
{\rm (b)}~The row vector $\pi^* := \pi^*_r$ is the stationary distribution of the ergodic $r$-spoke star chain with transition matrix $P^*$ satisfying, for $i = 1, \dots, r$,
\begin{align*}
P^*(i, 0) &=1-\gamma_i\quad\text{and}\quad P^*(i,i)=\gamma_i.\\
P^*(0, i) &= \frac{(1-\gamma_i) \pi^*(i)}{\pi^*(0)}\quad \text{and} \quad \quad P^*(0,0)=1-\frac{1}{\pi^*(0)}\sum_{i=1}^r (1-\gamma_i) 
\pi^*(i).
\end{align*}
\enl 
\begin{proof}
\ \smallskip \\
\indent
(a)~Fix $k \in \{0, \dots, r\}$.  Clearly $\pi^*_k(0) > 0$, so we begin by showing that $\pi^*_k(i) > 0$ for $i = 1, \dots, k$.  
Since $1 - \rho_i > 0$, we'll do this by showing that each factor in the product $\prod_{j \neq i}$ in~\eqref{p} is strictly positive.  Indeed, if $j > i$ this is clear because $0 < \rho_j < 1$.  If $j < i$, then we use
$$
\frac{1-\gamma_j - \rho_j (1-\gamma_i)}{\gamma_i - \gamma_j} = 
\frac{\rho_j(1- \gamma_i) -(1- \gamma_j)}{\gamma_j - \gamma_i} > 
\frac{\left(\frac{1-\gamma_j}{1-\gamma_i}\right) (1-\gamma_i) - (1-\gamma_j)}{\gamma_j - \gamma_i} = 0,
$$
where the inequality holds because $\lambda_j > \gamma_i$ by the interlacing condition.  To show $\sum_{i=1}^k \pi^*_k(i)=1$, we repeat the argument in the proof of Lemma 2.1 
in~\cite{MB} and include it for completeness.  Define 
\begin{equation}
\label{psidef}
\psi(s) := \prod_{i=1}^k \frac{1-\gamma_i+\rho_i s}{1-\gamma_i+s}.
\end{equation} 
Then $\psi(0)=1$, and we will show
\begin{align}
\label{psi}
\psi(s)&=\pi^*_k(0)+\sum_{i=1}^k \pi^*_k(i)\frac{1-\gamma_i}{1-\gamma_i+s}\ \text{for general s}\\
\nonumber
&=\sum_{i=0}^k \pi^*_k(i)\ \text{at }s=0,
\end{align}
which will complete the argument.  To show~\eqref{psi}, first set 
$$
f(s):=\prod_{j=1}^k(1-\gamma_j+\rho_j s),\quad
g(s):=\prod_{j=1}^k(1-\gamma_j+s),\quad
\tilde{f}(s):=f(s)-\left(\prod_{j=1}^k \rho_j\right)g(s).
$$
Note that $\tilde{f}(s)$ is a polynomial of degree $\leq k-1$ and  that 
$$\tilde{f}(-1+\gamma_i)=f(-1+\gamma_i),\quad i=1,\ldots,k.$$
Define 
$$h(s):=\sum_{i=1}^k \left(\pi^*_k(i)(1-\gamma_i)\prod_{j\neq i:1\leq j\leq k}(1-\gamma_j+s)\right).$$
A brief calculation yields 
$$\pi^*_k(i)(1-\gamma_i)=\frac{f(\gamma_i-1)}{g'(\gamma_i-1)},$$
and we see  that 
$$h(\gamma_i-1)=f(\gamma_i-1)=\tilde{f}(\gamma_i-1),\quad i=1,\ldots,k.$$
But $h(s)$, like $\tilde{f}(s)$, is a polynomial of degree $\leq k-1$, and so $h(s)=\tilde{f}(s)$ for all  $s$.  Finally, we see
\begin{align*}
\psi(s)&=\frac{f(s)}{g(s)}
   =\frac{1}{g(s)}\left[\left(\prod_{i=1}^k \rho_i \right)g(s)+\tilde{f}(s)\right]
   =\prod_{i=1}^k \rho_i +\frac{h(s)}{g(s)} \\
&=\pi^*_k(0)+\sum_{i=1}^k \pi^*_k(i) \frac{1-\gamma_i}{1-\gamma_i+s},
\end{align*}
establishing~\eqref{psi} and completing the proof of part~(a).
\smallskip

(b) Clearly, $P^*\vec{1}^T=\vec{1}^T$.  To show that $P^*$ is stochastic, we need only show that $P^*\geq 0$ entrywise.  This is clear except perhaps for the entry $P^*(0,0)$.  To see $P^*(0,0) > 0$, we first note that 
$P^*(0,0)=\mathrm{tr}\,P^*-\mathrm{tr}\,P^*_{0}$; Lemma 2.6 in~\cite{MB} then 
gives $\mathrm{tr}\,P^*-\mathrm{tr}\,P^*_{0} = \sum_{i=0}^r \lambda_i-\sum_{i=1}^r \gamma_i = \sum_{i = 0}^{r - 1} (\lambda_i - \gamma_{i + 1}) + \lambda_r > 0$.
Part~(a) establishes that $\pi^* = \pi^*_r$ is a distribution, and one sees immediately that $\pi^*$ satisfies the detailed balance equations for the transition matrix $P^*$.
\end{proof}

\begin{remark}
It would seem natural to define a $k$-spokes star chain with transition matrix ${P^*}^{(k)}$ and stationary 
distribution $\pi^*_k$ for general~$k$ just as is done for $k = r$ in \refL{L:p}.  However, it is then not clear whether 
${P^*}^{(k)}(0,0) \geq 0$.  Moreover, in our construction we use only the $P^*$ of \refL{L:p}(b) (with $k = r$).
\end{remark} 

Define $P^*_{\text{abs}}$ to be the chain $(X_t^*)_{t=0,1,\ldots}$  modified so that~$0$ is an absorbing state and note that 
$$\sigma(P^*_{\rm abs})=\lbrace1,\gamma_1,\ldots,\gamma_r\rbrace.$$  
We now begin to head towards \refT{T:ll}, which will show that $\mathcal{L}_{\pi^*}(T_0^*)=\mathcal{L}( \sum_{i=1}^r Y_i)$ for the $Y_i$'s described 
in \refT{T:Brown1d}.  To do this, we will construct a link $\Lambda_2$ between the absorbing star chain and a dual chain $(\widehat{X}_t)_{t=0,1,\ldots}$ for which the hitting time for state~$0$ 
is explicitly given as an independent sum of the modified Geometric random variables $Y_i$.

\begr
\label{R:star}
If the given chain \emph{is} already a star chain, then the star chain of \refL{L:p} is simply obtained by collapsing all leaves with the same one-step transition probability to state~$0$ into a single leaf.  This is established as Proposition~\ref{prop:star} in the Appendix, where it is also shown that the stationary probabilities collapse accordingly.  For example, suppose the given chain is the star chain with transition matrix 
\[ P = \left( \begin{array}{cccccc}
4/9 & 1/9 & 1/9 & 1/9 & 1/9 & 1/9 \\
1/6 & 5/6 & 0 & 0 & 0 & 0 \\
1/6 & 0 & 5/6 & 0 & 0 & 0 \\
2/9 & 0 & 0 & 7/9 & 0 & 0 \\
1/3 & 0 & 0 & 0 & 2/3 & 0 \\
1/3 & 0 & 0 & 0 & 0 & 2/3 \end{array} \right).\]
We see that
$\pi = \frac{1}{21}(6, 4, 4, 3, 2, 2)$ 
and that 
$$
\sigma(P) = \{1,\,5/6,\,0.8023,\,0.7303,\,2/3,\,0.1896\}, \ \  \sigma(P_0) = \{5/6,\,5/6,\,7/9,\,2/3,\,2/3\}.  
$$
The reduced set of eigenvalues of $P$ is $\{1,\,0.8023,\,0.7303,\,0.1896\}$ and the reduced set of eigenvalues of $P_0$ is 
$\{5/6,\,7/9,\,2/3\}$.  The star chain constructed in \refL{L:p} has three spokes with 
probabilities $1/6,\,2/9,\,1/3$ of moving to the hub in one step and respective stationary probabilities 
$8 / 21,\,3 / 21,\,4 / 21$ (with stationary probability $6 / 21$ at the hub). 
\enr 

The key to our construction will be the following ``spoke-breaking'' 
theorem.
\begt
\label{T:heart}
For each $i=1,\ldots, r$, the distribution $\pi^*_i \in \mathbb{R}^{r+1}$ can be represented as the mixture
\begin{equation}
\label{mix}
\pi^*_i = \rho_i \pi^*_{i - 1} + (1 - \rho_i) \nu_i
\end{equation}
of $\pi^*_{i - 1}$ and a probability distribution $\nu_i$ {\rm (}regarded as a row vector in $\mathbb{R}^{r+1}${\rm )} satisfying
\begin{equation}
\label{link}
\nu_i P^* = \gamma_i \nu_i + (1-\gamma_i) \pi^*_{i - 1}.
\end{equation}
\ent
\begin{proof}
Fix $i$.  Clearly there is a unique row vector $\nu \equiv \nu_i$ satisfying~\eqref{mix}, and it sums to unity because $\pi^*_i$ and 
$\pi^*_{i - 1}$ each do.  We will solve for~$\nu$ and see immediately that $\nu$ has nonnegative entries; indeed, we will show that~$\nu$ is given by
\begin{equation}
\label{nu}
\nu(j) =
\begin{cases}
\frac{1-\gamma_i}{1-\gamma_i - \rho_i (1-\gamma_j)} \pi^*_i(j) & \mathrm{if\ }1\leq j\leq i \\
0 & \mathrm{if\ }j = 0 \ \mathrm{or }\ j>i.
\end{cases}
\end{equation}
It will then be necessary only to prove that~$\nu$ satisfies~\eqref{link}.

We begin by establishing~\eqref{nu} for $j = 1, \dots, i - 1$.  (For $t = i - 1$ and $t = i$, the notation $\prod^t_{k \neq j}$ will be shorthand for the product over values~$k$ satisfying both 
$1 \leq k \leq t$ and $k \neq j$.)  In that case,
\begin{eqnarray*}
(1 - \rho_i) \nu(j)
&=& \pi^*_i(j) - \rho_i \pi^*_{i - 1}(j) \\
&=& (1 - \rho_j) \prod^i_{k \neq j} \frac{1-\gamma_k - \rho_k (1-\gamma_j)}{\gamma_j - \gamma_k} - 
\rho_i (1 - \rho_j) \prod^{i - 1}_{k \neq j} \frac{1-\gamma_k - \rho_k (1-\gamma_j)}{\gamma_j - \gamma_k} \\
&=& \pi^*_i(j) \left[1 - \rho_i \frac{\gamma_j - \gamma_i}{1-\gamma_i - \rho_i (1-\gamma_j)} \right] \\
&=& (1 - \rho_i) \frac{1-\gamma_i}{1-\gamma_i - \rho_i (1-\gamma_j)} \pi^*_i(j),
\end{eqnarray*}
as desired, where the first equality follows from~\eqref{mix}, and the second and third employ the 
formula~\eqref{p} both for $\pi^*_i$ and for $\pi^*_{i - 1}$.

For $j = i$ we calculate
$$
(1 - \rho_i) \nu(i) = \pi^*_i(i) - \rho_i \pi^*_{i - 1}(i) = \pi^*_i(i),
$$
\ie,
$$
\nu(i) = (1 - \rho_i)^{-1} \pi^*_i(i) = \frac{1-\gamma_i}{1-\gamma_i - \rho_i (1-\gamma_i)} \pi^*_i(i),
$$
again as desired.

For $j = 0$, \eqref{p} gives that
$$
(1 - \rho_i) \nu(0) = \pi^*_i(0) - \rho_i \pi^*_{i - 1}(0) = \prod_{k = 1}^i \rho_k - \rho_i \prod_{k = 1}^{i - 1} \rho_k = 0,
$$
once again as desired.  For $j>i$,~\eqref{nu} is clear because $\pi^*_i(j)=0=\pi^*_{i - 1}(j)$.

It remains to check that~$\nu$ satisfies~\eqref{link}.  Since both sides are vectors summing to~$1$ (on the left because $\nu$ is a probability distribution and $P^*$ is a transition kernel, and on the right because both~$\nu$ and $\pi^*_{i - 1}$ are probability distributions), we need only check $\mathrm{LHS}(j) = \mathrm{RHS}(j)$ for 
$j \neq 0$ (henceforth assumed).  We begin by calculating the state-$j$ entry of the LHS assuming $j\leq i$:
\begin{eqnarray*}
\mathrm{LHS}(j) 
&=& \sum_{k = 0}^r \nu(k) P^*(k, j) = \sum_{k = 1}^i \nu(k) P^*(k, j) \\
&=& \nu(j) P^*(j, j) = \left[ \frac{1-\gamma_i}{1-\gamma_i - \rho_i (1-\gamma_j)} \pi^*_i(j) \right] \times ( \gamma_j).
\end{eqnarray*}
On the other hand, using~\eqref{mix} we calculate
\begin{align*}
\mathrm{RHS} &= \pi^*_{i - 1} + \gamma_i \left( \nu - \pi^*_{i - 1} \right)\\
 &= \pi^*_{i - 1} + \gamma_i \left[ \nu - \rho^{-1}_i \left( \pi^*_i - (1 - \rho_i) \nu \right) \right] \\
 &= \rho^{-1}_i \left( \pi^*_i  - (1 - \rho_i) \nu \right)+ \frac{\gamma_i}{\rho_i} \left( \nu - \pi^*_i  \right).
 \end{align*}
 Therefore, for $j\leq i$ the $j$th entry of the RHS is \begin{align*}
 \mathrm{RHS}(j)&=\rho_i^{-1}\pi^*_i(j)\left[ 1-\frac{(1-\rho_i)(1-\gamma_i)}{1-\gamma_i - \rho_i (1-\gamma_j)}+\frac{\gamma_i\rho_i(1-\gamma_j)}{1-\gamma_i - \rho_i (1-\gamma_j)} \right]\\
 &=\rho_i^{-1}\pi^*_i(j)\left[ \frac{(1-\gamma_i)\rho_i\gamma_j}{1-\gamma_i - \rho_i (1-\gamma_j)} \right]\\
 &= \mathrm{LHS}(j).
\end{align*}
If $j>i$, then LHS$(j)=0=\mathrm{RHS}(j)$, finishing the proof that $\nu$ satisfies~\eqref{link}.
\end{proof}

The preceding \refT{T:heart} suggests the form for the chain $(\widehat{X}_t)_{t=0,1,2,\ldots}$ on $\lbrace 0,1,\ldots,r\rbrace$, where the times spent in state $j=0,1,2,\ldots,r$ in this chain are independent and distributed as the $Y_j$'s in 
\refT{T:Brown1d}.  Before proceeding to the construction in \refL{L:hat}, the next lemma provides some preliminaries.
\begl
 Let 
 $0< k\leq r$.  Let $\pih_k(j) := \rho_k\rho_{k-1}\ldots\rho_{j+1}(1-\rho_j)$ for all $1\leq j < k$, and let $\pih_k(k) := 1-\rho_k$.  Then 
 $\pih_k(j) \geq 0$ for $1 \leq j \leq k$, and $\sum_{j=1}^k \pih_k(j) =1-\prod_{i=1}^k \rho_i$. If we define $\pih_k(0) := \prod_{i=1}^k \rho_i$, then $\pih_k$ gives a probability distribution on $0,1,\ldots,k$.  
 \enl
The proof of this lemma is very easy.  Let us also adopt the convention $\pih_0 := \delta_0$.

 \begin{remark}
 \label{R:parallel}
Paralleling~\eqref{mix} in \refT{T:heart}, we have
$$\pih_k = \rho_k \pih_{k - 1} + (1-\rho_k)\delta_k\quad \mathrm{for}\  1\leq k\leq r.$$
\end{remark}

We are now ready to construct $(\widehat{X}_t)$:

\begl
\label{L:hat}
Let $(\widehat{X}_t)$ be the absorbing Markov chain with state space $\lbrace 0,\ldots,r\rbrace$ begun in distribution 
$\pih := \pih_r$, with transition matrix $\widehat{P}$ defined by
$$
\widehat{P}(i,j)= 
\left\{
\begin{array}{ll}
1 &\mbox{\rm if\ }0=j=i\\
\gamma_i &\mbox{\rm if\ }0<j=i\\
(1-\gamma_i)\cdot\pih_{i - 1}(j) &\mbox{\rm if\ }j<i\\
0 &\mbox{\rm if\ }j>i.
\end{array}
\right.
$$
Then 
\begin{enumerate}
\item[{\rm (a)}] If $Z_i$ is the time spent in state i {\rm (}including time 0{\rm )} by $(\widehat{X}_t)$ with initial distribution $\pih$ prior to hitting $0$, then $\mathcal{L}(Z_1,Z_2,\ldots,Z_r)=\mathcal{L}(Y_1,Y_2,\ldots,Y_r)$.
\item[{\rm (b)}] If\ $\widehat{T}_0$ is the hitting time of state 0 for the chain $(\widehat{X}_t)$ with initial distribution $\pih$,  then 
$\widehat{T}_0\Leq\sum_{i=1}^r Y_i.$
\end{enumerate}
\enl   
\begin{proof}
(a)~When 
viewed in the right light, the lemma is evident.  The chain moves downward through the state space $\lbrace0,1,\ldots,r\rbrace$, with ultimate absorption in state~$0$, and can be constructed by performing a sequence of $r$ independent Bernoulli trials $ W_r,\ldots,W_1$ with varying success probabilities $1-\rho_r,\ldots,1-\rho_1$, respectively.  If $W_i=0$, then the chain does not visit state $i$, whereas if $W_i=1$ then the amount of time spent in state $i$ is Geom$(1-\gamma_i)$ independent of the amounts of time spent in the other states.

A formal proof of part~(a) is not difficult but would obscure this simple construction and is therefore not included.

(b)~This is immediate from part~(a), since $\widehat{T}_0=\sum_{i=1}^r Z_i$.
\end{proof}
As the culmination of this subsection we exhibit an intertwining between $(\pi^*, P_{\mathrm{abs}}^*)$ and 
$(\pih, \widehat{P})$.   

\begt
\label{T:ll}
Let $\Lambda_2$ be defined as follows:
$$
\Lambda_2(0,:) := \delta_0, \qquad \Lambda_2(i,:):=\nu_i \mathrm{\ for\ }i=1,\ldots,r.
$$ 
Then
$(\pi^*, P_{\mathrm{abs}}^*)$ and $(\pih, \widehat{P})$ are intertwined by the link~$\gL_2$, which satisfies~\eqref{abs}; to wit,
\begin{align}
\label{a}
\Lambda_2 P^*_{\mathrm{abs}} &= \widehat{P}\Lambda_2,\\
\label{b}
\pi^* &= \pih \Lambda_2,\\
\label{c}
\gL_2 \delta_0^T &= \delta_0^T.
\end{align}
\ent
\begin{proof}
We begin by noting that~$\gL_2$ is stochastic because, as noted in \refT{T:heart}, each $\nu_i$ is a probability distribution.  

From \refT{T:heart} we have that $\pi^*_k = \rho_k \pi^*_{k - 1} + (1 - \rho_k) \nu_k$ for $1\leq k\leq r$,  and from \refR{R:parallel} we have the corresponding equations for $\pih_k$,  namely, $\pih_k = \rho_k \pih_{k - 1} + (1 - \rho_k) \delta_k$  for all $1\leq k\leq r$.  One can use these results to prove $\pi^*_k = \pih_k \Lambda_2$ for $k=0,1,\ldots,r$ by induction on $k$; in particular, \eqref{b} follows by setting $k = r$.  

To show~\eqref{a}, first observe $(\Lambda_2 P^*_{\mathrm{abs}}) (0,:)=\delta_0=(\widehat{P}\Lambda_2)(0,:)$.  Comparing $i$th rows for $1\leq i\leq r$, we see 
\begin{equation}
\label{Lambda2Pabs*}
(\Lambda_2 P^*_{\mathrm{abs}})(i,:)=\nu_i P^*_{\mathrm{abs}}=\gamma_i \nu_i + (1-\gamma_i) \pi^*_{i - 1} 
\end{equation} 
by ~\eqref{link} and the fact that $\nu_i(0)=0$ for all $i$.  Iterating ~\refT{T:heart}, we see for $i=1,\ldots,r$ that
\begin{align*}
\pi^*_i &=(1-\rho_i)\nu_i + \rho_i \pi^*_{i - 1}\\
&=(1-\rho_i)\nu_i + \rho_i\left[(1-\rho_{i-1})\nu_{i - 1} + \rho_{i-1} \pi^*_{i - 2}\right]\\
&=(1-\rho_i)\nu_i + \rho_i(1-\rho_{i-1})\nu_{i - 1} + \rho_i\rho_{i-1} \pi^*_{i - 2}\\
&=\cdots=\pih_i(i)\nu_i + \pih_i(i-1)\nu_{i - 1} + \cdots+\pih_i(1)\nu_1+\pih_i(0)\delta_0.
\end{align*}
So $\pi^*_i = \sum_{j=1}^i \pih_i(j)\nu_j + \pih_i(0)\delta_0$ for $i = 1, \dots, r$, and the same equation holds for $i = 0$ because $\pi^*_0 = \delta_0 = \pih_0$.  Applying this to equation~\eqref{Lambda2Pabs*} we find for $i = 1, \dots, r$ that
\begin{align*}
(\Lambda_2 P^*_{\mathrm{abs}})(i,:)&=\gamma_i \nu_i + \sum_{j=1}^{i-1}(1-\gamma_i)\pih_{i - 1}(j)\nu_j 
+ (1-\gamma_i) \pih_{i - 1}(0) \delta_0\\
&=(\widehat{P}\Lambda_2)(i,:),
\end{align*}
as desired, where at the last equality we have recalled $\Lambda_2(0, :) = \delta_0$.

Finally, \eqref{c} asserts that the $0$th column of $\gL_2$ is $\delta_0^T$.  This follows from the definition of $\gL_2$, since it has already been noted at~\eqref{nu} that $\nu_i(0) = 0$ for $i = 1, \dots, r$.
\end{proof}

\subsection{Quasi-link to the star chain}
\label{S:starlink}

The main result of this subsection is \refT{T:link2}, which provides a quasi-link between the \emph{absorbing} transition matrices 
$P_{\text{\rm abs}}$ and $P^*_{\text{\rm abs}}$ corresponding to the given chain and the star chain, respectively.  We begin with a linear-algebraic lemma. 

\begl
The matrix $P_{\mathrm{abs}}$ has $n + 1$ linearly independent left eigenvectors.  Its multiset of $n + 1$ eigenvalues is $\lbrace 1, \eta_1, \dots, \eta_n\rbrace$.
\enl
\begin{proof}
Recall that $\sigma(P_0)=\lbrace\eta_1, \dots, \eta_n\rbrace$.  Recall also that $D_0 = \mathrm{diag}(\pi_1, \ldots, \pi_n)$ and that $S_0 = D_0^{1/2} P_0 D_0^{-1/2}$ is a symmetric matrix.  Let $\widetilde{U}$ be an $n$-by-$n$ orthogonal matrix whose rows are orthonormal left eigenvectors of $S_0$, so that $\widetilde{U} S_0 \widetilde{U}^T=\text{diag}(\eta_1,\eta_2,\ldots,\eta_n)$.  Then the rows (denoted $u_1, \dots, u_n$) of the $n$-by-$n$ matrix $U:=\widetilde{U} D_0^{1/2}$ are left eigenvectors of $P_0$ with respective eigenvalues $\eta_1,\ldots,\eta_n$.  For $i=1,\ldots,n$, define the scalar
$$w_i:=\frac{(0|u_i) P(:,0)}{\eta_i-1};$$
then $(w_i|u_i)P_{\text{abs}}=\eta_i(w_i|u_i)$ and $\eta_i\in\sigma(P_{\text{abs}})$.  
Finally, $\delta_0 P_{\text{abs}}=\delta_0$.  The $n+1$ eigenvectors $\delta_0$ and $(w_i|u_i)$ for $i=1,\ldots,n$ are clearly linearly independent, and our proof is complete.  
\end{proof}

Note that 
$$(w_i|u_i) \vec{1}^T=(w_i|u_i)P_{\text{abs}} \vec{1}^T=\eta_i(w_i|u_i) \vec{1}^T$$
and $\eta_i<1$, implying that $(w_i|u_i) \vec{1}^T=0$ and $w_i=-u_i \vec{1}^T$. 

Let $n_i$ denote the algebraic (also geometric) multiplicity of the eigenvalue $\gamma_i$ as an eigenvalue of $P_0$ (here we are working with the reduced set of eigenvalues again).  Relabel the eigenvectors corresponding to $\gamma_i$ by $u_1^i,\ldots,u_{n_i}^i$.  Note that, when viewed as an eigenvalue of $P_{\text{abs}}$, $\gamma_i$ has algebraic (also geometric) multiplicity $n_i$, with corresponding eigenvectors $(-u_1^i \vec{1}^T|u_1^i),\ldots,(-u_{n_i}^i \vec{1}^T|u_{n_i}^i)$.  In the next theorem we construct our 
$(r+1)$-{by}-$(n+1)$ quasi-link $\Lambda_1$ between  $(\pi, P_{\text{abs}})$ and $(\pi^*, P^*_{\text{abs}})$.
\begt
\label{T:link2}
There exists a quasi-link $\Lambda_1$ providing a quasi-intertwining between $(\pi, P_{\text{\rm abs}})$ and 
$(\pi^*, P^*_{\text{\rm abs}})$ and satisfying~\eqref{abs}, \ie, a matrix~$\gL_1$ with rows summing to~$1$ such 
that 
\begin{align}
\label{link21}
\pi &= \pi^* \Lambda_1,\\
\label{link22}
\Lambda_1 P_{\text{\rm abs}}&=P^*_{\text{\rm abs}}\Lambda_1,\\
\label{link23}
\gL_1 \delta_0^T &= \delta_0^T.
\end{align} 
\ent
%
\begin{proof}
If row $i$ of $\Lambda_1$ is denoted by $x_i$ for $i=0,...,r$, then for~\eqref{link22} we require
$$x_0 P_{\text{abs}}=x_0; \hspace{5mm} x_i P_{\text{abs}}=(1-\gamma_i)x_0+\gamma_i x_i\hspace{5mm}i=1,\ldots,r.$$
This forces $x_0 = \delta_0$ and $$x_i(P_{\mathrm{abs}}-\gamma_i I)=(1-\gamma_i)\delta_0, \quad i=1,\ldots,r.$$  Therefore, for 
$\Lambda_1 P_{\text{abs}}=P^*_{\text{abs}}\Lambda_1$ to hold, we necessarily set
$$x_i=\delta_0+\sum_{j=1}^{n_i} c_j^i\ (-u_j^i \vec{1}^T|u_j^i),\quad i=1\ldots,r$$ where the $c_j^i$'s are soon-to-be-determined real constants.  

For \emph{any} choices of $c_j^i$'s above we have that the rows of~$\gL_1$ sum to unity and 
$\Lambda_1 P_{\text{abs}}=P^*_{\text{abs}}\Lambda_1$, but it remains to be shown that we can define $c_j^i$'s so 
that~\eqref{link21} holds.  The difficulty is that  there may exist values 
$\eta_i\in\sigma(P_0)$ such that $\eta_i\neq\gamma_j$ for any $j=1,\ldots,r$.  However, we will show in the next lemma that~$\pi$ is in the span of the eigenvectors corresponding to the remaining eigenvalues, and that will complete our proof of~\eqref{link21}.

To prove~\eqref{link23}, we use \eqref{link21}--\eqref{link22} to get 
$\pi P_{\text{abs}}^t = \pi^*\Lambda_1 P_{\text{abs}}^t= \pi^* P^{*t}_{\text{abs}}\Lambda_1$; we find [using $\Lambda_1(0, 0) = 1$] that the $0$th entry of this vector is 
\begin{align*}
\PP_{\pi}(T_0\leq t)&=\sum_i \pi^*(i) \sum_j P_{\mathrm{abs}}^{*t}(i,j) \Lambda_1(j,0)\\
&=\pi^*(0)+\sum_{i\neq 0} \pi^*(i)[P_{\mathrm{abs}}^{*t}(i,0)+P_{\mathrm{abs}}^{*t}(i,i)\Lambda_1(i,0)]\\
&=\pi^*(0)+\sum_{i\neq 0} \pi^*(i)[1+P_{\mathrm{abs}}^{*t}(i,i)(\Lambda_1(i,0)-1)]\\
&=\pi^*(0)+\sum_{i\neq 0} \pi^*(i)[1+\gamma_i^t(\Lambda_1(i,0)-1)]\\
&=1+\sum_{i = 1}^r \pi^*(i) \gamma_i^t (\Lambda_1(i,0)-1).
\end{align*} 
We also have from~\eqref{hit1} in the proof of the next lemma that $\PP_{\pi}(T_0\leq t)=1-\sum_{i=1}^r \pi^*(i) \gamma_i^t$.  Therefore $\Lambda_1(i,0)=0$ for $i>0$, and \eqref{link23} follows.
\end{proof}

\begl
There exist 
real constants $c_j^i$ such that $\pi=\pi^* \Lambda_1$.
\enl 
\begin{proof}
We will make use of the fact that 
\begin{equation}
\label{hit1}
\PP_{\pi}(T_0>t)=\sum_{j=1}^r \pi^*(j) \gamma_j^t, \quad t = 0, 1, \dots,
\end{equation} 
which follows from its continuous-time analogue, 
equation~(1.1) in~\cite{MB},
using \refL{L:continuous-discrete}.
[That analogue is established using the fact that the function~$\psi$ in our equations \eqref{psidef}--\eqref{psi} is the Laplace transform of $T_0$ for the stationary continuized chain; see~\cite{MB} for further details.] 
Define 
$$
\pit := (\pi(1),\ldots,\pi(n))\in\mathbb{R}^n;
$$
we would use the notation $\pi_0$ to indicate this deletion of the $0$th entry from~$\pi$ except that it conflicts with our notation for the initial distribution of the given chain.
We then have that $P_{\pi}(T_0>t)= \pit P_0^t \vec{1}^T$.  Using the spectral representation of $P_0$ we find for $t \geq 0$ that
\begin{equation}
\label{hit2}
\PP_{\pi}(T_0>t)=\sum_{i=1}^n \sum_{j=1}^n\sum_{k=1}^n \sqrt{\pi(i)\pi(j)}\widetilde{U}(k,i)\widetilde{U}(k,j) \eta_k^t 
= \sum_{k = 1}^n q_k \eta_k^t.
\end{equation}
Here $\qq = (q_1, \dots, q_n) =(\pit^{1/2} \widetilde{U}^T)^2$, where both the nonnegative square root and the square are in the Hadamard sense.  In particular, $q_k \geq 0$ for all $k=1,\ldots,n$.  Comparing \eqref{hit1} and \eqref{hit2}, it is clear that if $\eta_i\neq\gamma_j$ for every $j=1,\ldots,r$, then $q_i=0$.  
Again comparing \eqref{hit1} and \eqref{hit2}, for each $\gamma_j$ there is an $\eta_k=\gamma_j$ such that the coefficient of $\eta_k^t$ in \eqref{hit2},  namely $q_k$, is strictly positive.    
Now $\qq=(\pit^{1/2}\widetilde{U}^T)^2$ equals the Hadamard square $(\pit D_0^{-1/2} \widetilde{U}^T)^2$.  
We can therefore choose $R$, a diagonal matrix with $\pm 1$ along the diagonal, such that $\pit = \qq^{1/2} R (\tilde{U} D_0^{1/2})=\qq^{1/2}R U$; here $\qq^{1/2}$ is the Hadamard nonnegative square root of~$\qq$.  Relabel the entries of the vector~$\qq$ (and of~$R$) so that
$$\pit=\sum_{i=1}^r \sum_{j=1}^{n_i} r_j^i (q^i_j)^{1/2} u_j^i.$$
Letting $c_j^i=r_j^i (q_j^i)^{1/2}/\pi^*(i)$ yields
$$\pit=\sum_{i=1}^r \sum_{j=1}^{n_i} \pi^*(i)\,c_j^i\,u_j^i.$$
It remains only to show that for this choice of $c_j^i$'s we have 
$$
\pi(0)=1+\sum_{i=1}^r \sum_{j=1}^{n_i} \pi^*(i) c_j^i (-u^i_j \vec{1}^T).
$$
This is immediate 
from 
$$
1-\pi(0)=\pit \vec{1}^T=\qq^{1/2} R U\vec{1}^T=\sum_{i=1}^r \sum_{j=1}^{n_i} r_j^i(q^i_j)^{1/2}(u_j^i\ \vec{1}^T).
\vspace{-.4in}
$$ 
\end{proof}
\vspace{3mm}
Our 
construction of~$\Lambda_1$ uses the eigenvectors of $P_{\text{abs}}$; the entries of these eigenvectors are not all nonnegative, and as a result neither (in general) are the entries of~$\Lambda_1$.  In the special case that the given chain is a star chain, the quasi-link $\Lambda_1$ is a \emph{bona fide} link.  For example, for the chain considered in \refR{R:star} the quasi-link $\Lambda_1$ is easily seen to be the link
\[ \Lambda_1 = \left( \begin{array}{cccccc}
1 & 0 & 0 & 0 & 0 & 0 \\
0 & 1/2 & 1/2 & 0 & 0 & 0 \\
0 & 0 & 0 & 1 & 0 & 0 \\
0 & 0 & 0 & 0 & 1/2 & 1/2 \end{array} \right).\]

\begr 
\label{R:unique} 
If $r=n$ (\ie, the reduced spectra are the same as the unreduced spectra), then it is not hard to show that the quasi-link 
$\gL_1$ of \refT{T:link2}
is uniquely determined.
\enr  
\subsection{The big link~$\Lambda$}
\label{S:biglink}
Combining the quasi-link $\Lambda_1$ of \refT{T:link2} between $(\pi, P_{\mathrm{abs}})$ and $(\pi^*, P^*_{\mathrm{abs}})$ and the link $\Lambda_2$ of \refT{T:ll} between $(\pi^*, P^*_{\mathrm{abs}})$ and $(\pih, \widehat{P})$, we obtain the desired quasi-link $\Lambda = \Lambda_2 \Lambda_1$ between $(\pi, P_{\mathrm{abs}})$ and $(\pih, \widehat{P})$.  

\begt
\label{T:biglink}
Let $\Lambda:=\Lambda_2 \Lambda_1$.  Then~$\gL$ is a quasi-link providing a quasi-inter\-twining of $(\pi, P_{\mathrm{abs}})$ and 
$(\pih, \Ph)$, and therefore $ \Lc_{\pi} T_0 = \Lc_{\pih} \Th_0$.
\ent
\begin{proof}
This follows from \refR{R:prod} and the discussion in \refS{S:quasi}.
\end{proof}

If $\gL$ is stochastic, then we have a link between $P_{\rm abs}$ and $\widehat{P}$ and we can use the discussion following Definition~\ref{D:twine} to construct a sample path of $(\widehat{X}_t)$ given a realization of $(X_t)$.  However, it's easy to find examples showing that~$\gL$ is not nonnegative in general.  

The discussion
preceding \refR{R:unique} shows that $\gL$ is a link if the given chain~$X$ is a star chain.  More generally, $\gL$ is a link if the given chain is a ``block star chain'', defined as follows:  Choose positive numbers $b_0, \dots, b_k$ summing to unity and $0 < \pi_0 \leq 1$.  For $i = 1, \dots, k$, let 
$c_i := \pi_0 b_i$ and let $Q_i$ be an ergodic and reversible Markov kernel with stationary probability mass function $\pi_i$.  Let~$P$ be the following special case 
of~\eqref{block}:
\[
P = 
\left( 
\begin{array}{cccccc}
b_0 & b_1 \pi_1 & b_2 \pi_2 & \hdots & b_k \pi_k \\
c_1 \vec{1}^T & (1 - c_1) Q_1 & 0  &\hdots & 0 \\
c_2 \vec{1}^T & 0 & (1 - c_2) Q_2  & \hdots & 0 \\
\vdots & \vdots &\vdots  & \ddots  & \vdots  \\
c_k \vec{1}^T & 0&0 & \hdots & (1 - c_k) Q_k
\end{array}
\right);
\] 
it is easily checked that~$P$ is ergodic and reversible with stationary distribution equal to the concatenated row vector $(\pi_0 + k)^{-1} (\pi_0 | \pi_1 | \cdots | \pi_k)$, and that the reduction of spectra described in \refS{S:Brown1} results in 
$\{\gamma_1, \dots, \gamma_r\}$ being some subset of distinct elements from $\{1 - c_1, \dots, 1- c_k\}$. 
If, for example,  $r = k$, then $\gL_1$ is the matrix~\eqref{blocklink}, where $\mu_0 = (1)$ is 1-by-1 and we recall for 
$1 \leq j \leq k$ that $\mu_j\ (=\pi_j)$ is the quasi-stationary distribution for the $j$th diagonal block $(1-c_j)Q_j$ of~$P$;
hence $\gL_1$ is a link (and so, then, is $\gL = \gL_1 \gL_2$).  We are not aware of other interesting cases 
where~$\gL$ is guaranteed to be a link, but the key is to arrange, as for block star chains, for $P_0$ to have nonnegative eigenvectors corresponding to eigenvalues $\gamma_1 \dots, \gamma_r$.

\begr
Is there a \emph{unique} quasi-link~$\gL$ which, like the one constructed in \refT{T:biglink}, satisfies $\gL \gd_0^T = \gd_0^T$ and provides a quasi-intertwining of $(\pi, P_{\rm abs})$ and $(\pih, \Ph)$?  We do not know the answer in general, but if $r=n$, then the answer is affirmative by \refR{R:unique} and the invertibility of $\Lambda_2$.
\enr 

\section{Another representation for 
hitting times from stationarity}
\label{S:V}
Our final application of the strategy outlined in \refS{S:strategy} will provide a stochastic construction for an alternative characterization of the hitting-time distribution from stationarity first proved by Mark Brown [personal communication] in an unpublished technical report.  A published version of a special case can be found in~\cite{MBHT2}.  Our construction here is notable in that it will provide a generalization (to not necessarily reversible chains) of the discrete-time analogue of Brown's original result, and it is by applying our strategy that we discovered the generalization.

Brown's original theorem is the following, in which~$0$ is an arbitrary fixed state.
\begt[Mark Brown] \label{T:MBc}
Consider an ergodic time-reversible finite-state con\-tin\-u\-ous-time Markov chain with stationary distribution~$\pi$.  Let~$V$ be a random variable with
$$
\PP(V > t) = \frac{P_{0 0}(t) - \pi(0)}{1 - \pi(0)}, \quad 0 \leq t < \infty.
$$
Let $V_1, V_2, \dots$ be iid copies of~$V$, and let~$N$ be independent of the sequence 
$(V_i)$ with $N+1$ distributed Geometric with success probability~$\pi(0)$:
$$
\PP(N = k) = \pi(0) [1 - \pi(0)]^k, \quad k = 0, 1, \dots.
$$
Then the distribution $\Lc_{\pi} T_0$ of the nonnegative hitting time $T_0$ of~$0$ from a stationary start is the distribution of 
$\sum_{i = 1}^N V_i$.
\ent

We will focus on the following discrete-time analogue.  As in \refS{S:lace}, analogues of all of our results can be established in the continuous-time setting as well, but we have chosen discrete time for convenience and ease of understanding.

\begt \label{T:MBd}
Consider an ergodic time-reversible finite-state discrete-time Markov chain with stationary distribution~$\pi$.
Assume that $P^t(0, 0)$ is nonincreasing  in~$t$.  Let~$V$ be a random variable with
$$
\PP(V > t) = \frac{P^t(0, 0) - \pi(0)}{1 - \pi(0)}, \quad t = 0, 1, \dots.
$$
Let $V_1, V_2, \dots$ be iid copies of~$V$, and let~$N$ be independent of the sequence 
$(V_i)$ with $N + 1$ distributed Geometric with success probability~$\pi(0)$:
$$
\PP(N = k) = \pi(0) [1 - \pi(0)]^k, \quad k = 0, 1, \dots.
$$
Then the distribution $\Lc_{\pi} T_0$ of the nonnegative hitting time $T_0$ of~$0$ from a stationary start is the distribution of 
$\sum_{i = 1}^N V_i$.
\ent

The assumption in \refT{T:MBd} that $P^t(0, 0)$ is nonincreasing in~$t$ is met, for example,
if the chain is time-reversible and all the eigenvalues of the one-step transition matrix~$P$
are nonnegative.  {\bf However, we do not need to assume reversibility to follow our approach}, so \refT{T:MBd} (and likewise \refT{T:MBc}) is true without that assumption.  For a non-reversible scenario in which the nonincreasingness assumption is satisfied, see \refR{R:sep} and the paragraph preceding~it.

Following our strategy, we aim to provide a sample-path intertwining of the given chain~$X$ in \refT{T:MBd} with a chain $\Xh$  (with, say, initial distribution~$\pih_0$ and transition matrix~$\Ph$) for which the hitting time~$\Th_0$ has (for each sample path) a clear decomposition $\sum_{i = 1}^N V_i$ as in the theorem.   
As in our earlier application, we can treat~$0$ as an absorbing state for the given chain, whose one-step transition matrix we then denote by $P_{\mathrm{abs}}$.  We thus wish to find $(\pih_0, \Ph)$ and a link (or at least quasi-link) $\Lambda$ such that
$\pi = \pih_0  \Lambda$ and $\Lambda P_{\mathrm{abs}} = \widehat{P} \Lambda$.
The chain~$\Xh$ we will construct has state 
space $\{0, 1, \dots \}$.  Although the state space is infinite, this gives no difficulties as the needed intertwining results from~\cite{DFSST} apply just as readily to Markov chains with countably infinite state spaces.
First we construct our~$\Lambda$.

Suppose the given chain has
state space $\{0, 1, \dots, n\}$.  We adopt notation that highlights the special role of state~$0$.  
Let $\pi = (\pi(0) |\pi_{-0}) \in \mathbb{R}^{n + 1}$ with $\pit \in \mathbb{R}^n$, and similarly 
let $P^{i-1}(0,:)=(P^{i-1}(0,0)\ |P^{i-1}(0,:)_{-0})\in\mathbb{R}^{n + 1}$.
For
$i = 1 ,2, 3, \ldots$, define
$$
\mu_i := \left(0 \left| \frac{P^{i-1}(0,0)\pi_{-0}-\pi(0)P^{i-1}(0,:)_{-0}}{P^{i-1}(0,0)-\pi(0)} \right. \right) \in \mathbb{R}^{n + 1}.
$$

\begl
\label{L:mu}
With $\mu_i$ defined above, we have for $i>0$ that
$$
\mu_i P=q_i \pi +(1-q_i) \mu_{i+1},
$$
where
$$q_i:=\frac{P^{i-1}(0,0)-P^{i}(0,0)}{P^{i-1}(0,0)-\pi(0)} \in [0,1).$$
\enl
\begin{proof}
First note
$$\mu_i P=\frac{P^{i-1}(0,0)(0|\pi_{-0})P-\pi(0) (0|P^{i-1}(0,:)_{-0})P}{P^{i-1}(0,0)-\pi(0)}.$$  
Now 
$(0|\pi_{-0})P = \pi-\pi(0)P(0,:)$, and similarly  
\begin{align*}
(0|P^{i-1}(0,:)_{-0})P 
&= P^{i-1}(0,:)P-P^{i-1}(0,0)P(0,:)\\
&= P^i(0,:) - P^{i-1}(0,0)P(0,:);
\end{align*}
hence
\begin{align*}
\lefteqn{\hspace{-.2in}P^{i-1}(0,0)(0| \pi_{-0})P - \pi(0) (0| P^{i-1}(0,:)_{-0})P} \\
&= P^{i-1}(0,0)\pi-\pi(0)P^{i}(0,:)\\
&=P^{i-1}(0,0)\pi-(P^i(0,0)\pi(0)|0)-(0|\pi(0)P^{i}(0,:)_{-0})\\
&=P^{i-1}(0,0)\pi-P^{i}(0,0)\pi+P^{i}(0,0)(0|\pi_{-0})-(0|\pi(0)P^{i}(0,:)_{-0})\\
&=[P^{i-1}(0,0)-P^{i}(0,0)]\pi+(0|P^i(0,0)\pi_{-0} - \pi(0)P^{i}(0,:)_{-0}).
\end{align*}
Letting 
$$q_i:=\frac{P^{i-1}(0,0)-P^{i}(0,0)}{P^{i-1}(0,0)-\pi(0)},$$ it follows that
$\mu_i P 
= q_i \pi 
+ (1-q_i)\mu_{i+1},
$
as desired.
\end{proof}
This lemma suggests the form for $\widehat{P}$ and $\Lambda$.  Let $\widehat{X}$ have state space $\lbrace 0,1,2,\ldots\rbrace$. Define the transition kernel~$\widehat{P}$ by setting $\widehat{P}(0,0) := 1$ and, for $i > 0$,
$$
\widehat{P}(i,0) :=\pi(0) q_i, \qquad
\widehat{P}(i,1) :=[1-\pi(0)] q_i, \qquad
\widehat{P}(i,i+1) :=1-q_i;
$$
we set $\Ph(i, j) := 0$ for all other pairs $(i, j)$.  As the following lemma shows, the hitting time $\Th_0$ for this chain~$\widehat{X}$ has a simple decomposition as a sum of Geometrically many iid copies of~$V$. 
\begl
Let~$\Xh$ have initial distribution $\pih_0:=\pi(0)\delta_0+[1-\pi(0)]\delta_1$ and one-step transition matrix~$\Ph$.  Then there exist random variables~$N$ and $V_1, V_2, \dots$ with joint distribution as in \refT{T:MBd} such that {\rm (}for every 
sample path{\rm )}
$\Th_0 = \sum_{i = 1}^N V_i$.
\enl

\begin{proof}
Let $N \geq 0$ denote the number of visits to state~$1$; and for $i = 1, \dots, N$, let $V_i$ denote the highest state reached in the time interval $[\tau_i, \tau_{i + 1})$, where $\tau_i$ denotes the epoch of $i$th visit to state~$1$.  Then all of the assertions of the lemma are clear; it is perhaps worth noting only that for $t = 0, 1, \dots$ we have
$$
\PP(V_1 > t) = \prod_{i = 1}^t (1 - q_i) = \frac{P^t(0,0)-\pi(0)}{1-\pi(0)}.
\vspace{-.35in} 
$$
\end{proof}
\vspace{.15in}

Define~$\Lambda$ by setting $\Lambda(0,:) := \delta_0$ and $\Lambda(i,:) := \mu_i$ for $i>0$.  Note that~$\Lambda$ has infinitely many rows, each of which is in $\mathbb{R}^{n + 1}$.
We then have the following theorem whose proof is almost immediate from the definitions and \refL{L:mu}.
\begt
\label{T:quasi}
The quasi-link $\Lambda$ provides a quasi-intertwining of $(\pi, P_{\mathrm{abs}})$ and $(\pih_0, \Ph)$ and satisfies~\eqref{abs}, and therefore $\Lc_{\pi} T_0 = \Lc_{\pih_0} \Th_0$. 
\ent
\begin{proof}
It is easily checked that each row of~$\Lambda$ sums to unity.  Further, for $i > 0$ using the observations that $\mu_i(0) \equiv 0$ and $\mu_1 = \left(0 \left| \frac{\pi_{-0}}{1-\pi(0)} \right. \right)$, one finds readily that the $i$th row of 
$\Lambda P_{\mathrm{abs}}$ is $\mu_i P = q_i \pi + (1 - q_i) \mu_{i + 1}$, which is the $i$th row of $\Ph \Lambda$.  The $0$th rows are both $\delta_0$, so we conclude $\Lambda P_{\mathrm{abs}} = \Ph \Lambda$.  Similarly, $\pi = \pih_0 \Lambda$.  Finally, since 
$\mu_0(0) = 1$ and $\mu_i(0) = 0$ for $i > 0$, we have $\gL \delta_0^T = \delta_0^T$, which is~\eqref{abs}.  The equality of 
hitting-time laws then follows from the discussion in \refS{S:quasi}. 
\end{proof}
Note that~$\Lambda$ is a link (in which case sample-path linking is possible) if and only if for every $t \geq 0$ the $t$-step transition probability $\Pt^t(i, 0)$ is maximized when $i = 0$; here~$\Pt$ is the time-reversed transition matrix $\Pt(i, j) := \pi(j) P(j, i) / \pi(i)$.  A sufficient condition for this is that the state space is partially ordered, $0$ is either a top element or a bottom element, and~$\Pt$ is stochastically monotone.
\begr
\label{fastsst}
The intertwining constructed in Lemma \ref{L:mu} and Theorem \ref{T:quasi} can be related to the fastest strong stationary time construction of \cite{ADFSST} and the corresponding strong stationary dual constructed in Example~2.6 of \cite{DFSST}.  In the interest of brevity, we omit an explanation of the connection.
\enr

\begr
\label{R:sep}
We claim that if~$P$ is such that $\Pt^t(i, 0)$ is maximized for every~$t$ when $i = 0$, then~$P$ automatically satisfies the assumption in \refT{T:MBd} that $P^t(0, 0)$ is nonincreasing in~$t$.  To see this, consider the chain~$X$ with transition matrix~$P$ started in distribution $\mu_1 = [1 - \pi(0)]^{-1} [\pi - \pi(0) \delta_0]$.  Then, for any state~$i$,
$$
\frac{\PP(X_t = i)}{\pi(i)} = \frac{\pi(i) - \pi(0) P^t(0, i)}{\pi(i) [1 - \pi(0)]} = \frac{1 - \Pt^t(i, 0)}{1 - \pi(0)}.
$$
If $s(t)$ is the separation 
of the chain at time~$t$, then $1-s(t)$ equals the minimum of this ratio over~$i$, namely, $[1 - P^t(0, 0)] / [1 - \pi(0)]$.  It is well known 
(\eg,\ \cite[Chapter~9]{ALDF}) that separation is nonincreasing in~$t$, so $P^t(0, 0)$ is nonincreasing. 
\enr
\bigskip

{\bf Acknowledgment.}  We thank Mark Brown for stimulating discussions and the anonymous referee for helpful comments.

\bibliographystyle{plain}
\bibliography{mybib}

\appendix

\section{$P^*$ when~$P$ is a star chain}
In Remark~\ref{R:star} it is claimed that if
the given chain~$P$ is a star chain, then the star chain of \refL{L:p} is simply obtained by collapsing all leaves with the same one-step transition probability to state~$0$ into a single leaf.  More precisely, we establish the following:
\begin{proposition}
\label{prop:star}
Let~$P$ be the transition matrix of an ergodic star chain with hub at~$0$.  If for each $\gamma_i$ in the reduced set of eigenvalues of $P_0$ we define 
$$
m(i):=\{j\in[n]:\eta_j=\gamma_i\},
$$
then 
$P^*(0,i) = \sum_{j\in m(i)} P(0,j)$.
\end{proposition}
\begin{proof}
Define $H:=\text{diag}(\eta_1,\ldots,\eta_n)$ and 
\begin{align*}
x &:= (P(0,1),\ldots,P(0,n)), \\
y &:=(1-\eta_1,\ldots,1-\eta_n),
\end{align*} 
so that 
\[ 
P = \left( \begin{array}{c|c}
P(0,0) & x \\ \hline
y^T & H  \end{array} \right).
\]
By the standard formula for the determinant of a partitioned matrix (\eg,\ \cite[Section~0.8.5]{HJ}),
if~$t$ is not in the spectrum $\{\eta_1, \dots, \eta_n\}$ of~$H$ then we find
\begin{ignore}
{
A simple calculation yields
\begin{align*}
&\det(tI-P)=\det\left( \begin{array}{c|c}
t-P(0,0) & -x \\ \hline
-y^T & tI-H  \end{array} \right)\\
&=\det \left[ \left( \begin{array}{c|c}
1 & x(tI-H)^{-1} \\ \hline
0 & I  \end{array} \right)
\left( \begin{array}{c|c}
t-P(0,0) & -x \\ \hline
-y^T & tI-H  \end{array} \right)
\left( \begin{array}{c|c}
1 & 0 \\ \hline
(tI-H)^{-1}y^T & I  \end{array} \right) \right] \\
&=\det\left( \begin{array}{c|c}
t-P(0,0)-x(tI-H)^{-1}y^T & 0 \\ \hline
0 & tI-H  \end{array} \right)\\
&=(t-P(0,0)-x(tI-H)^{-1}y^T)\det(tI-H).
\end{align*}
}
\end{ignore}
\begin{equation}
\label{chP}
\det(tI-P) = [t-P(0,0)-x(tI-H)^{-1}y^T] \det(tI-H)
\end{equation}
for the characteristic polynomial of~$P$.
Analogously, define $\Gamma:=\text{diag}(\gamma_1,\ldots,\gamma_r)$ and
\begin{align*}
x^* &:= (P^*(0,1),\ldots,P^*(0,r)), \\
y^* &:= (1-\gamma_1, \ldots,1-\gamma_r);
\end{align*}
if~$t$ is not in the spectrum $\{\gamma_1, \dots, \gamma_r\}$ of~$\Gamma$, then we find
\begin{equation}
\label{chP^*}
\det(tI-P^*)=[t-P^*(0,0)-x^*(tI-\Gamma)^{-1}{y^*}^T] \det(tI-\Gamma)
\end{equation}
for the characteristic polynomial of~$P^*$.

Note that\begin{align}
\nonumber
P(0,0) &= \tr P -\tr H = \sum_{i = 0}^n \theta_i - \sum_{i = 1}^n \eta_i \\
\label{p00s}
&= \sum_{i = 0}^r \lambda_i - \sum_{i = 1}^r \gamma_i  = \tr P^*-\tr \Gamma = P^*(0,0),
\end{align}
where the third equality is a result of the eigenvalue reduction procedure discussed in \refS{S:Brown1} and the fourth equality is from Lemma~2.6 in~\cite{MB}.
Similarly, for all 
$t\notin\{\eta_1,\ldots,\eta_n\}$ we have 
\begin{equation}
\label{detratios}
\frac{\det(tI-P)}{\det(tI-H)} = \frac{\det(tI-P^*)}{\det(tI-\Gamma)}.
\end{equation}
Therefore, for all $t\notin\{\eta_1,\ldots,\eta_n\}$ we have
\begin{equation}
\label{eq:starex}
\sum_{i=1}^n P(0,i) \frac{1-\eta_i}{t-\eta_i}
= \sum_{i=1}^r P^*(0,i) \frac{1-\gamma_i}{t-\gamma_i},
\end{equation}
because using definitions of $H, x, y, \Gamma, x^*, y^*$ and equations \eqref{chP}--\eqref{detratios} we find
\begin{align*}
\sum_{i=1}^n P(0,i) &\frac{1-\eta_i}{t-\eta_i}
=x(tI-H)^{-1}y^T
= t - P(0, 0) - \frac{\det(tI-P)}{\det(tI-H)} \\
&= t - P^*(0, 0) - \frac{\det(tI-P^*)}{\det(tI-\Gamma)}
= x^*(tI-\Gamma)^{-1}{y^*}^T
= \sum_{i=1}^r P^*(0,i) \frac{1-\gamma_i}{t-\gamma_i}.
\end{align*}

Rewrite~\eqref{eq:starex} as
$$\sum_{i=1}^r P^*(0,i) \frac{1-\gamma_i}{t-\gamma_i}=\sum_{i=1}^r\left(\sum_{j\in m(i)} P(0,j)\right) \frac{1-\gamma_i}{t-\gamma_i}.$$
Since $\gamma_1, \dots, \gamma_r$ are distinct,  it follow easily that
$P^*(0,i)=\sum_{j\in m(i)} P(0,j)$ for $i = 1, \dots, r$, as desired.
\end{proof}

Let~$\pi$ be the stationary distribution for~$P$. 
Using the formula for $P^*(0, i)$ provided by Proposition~\ref{prop:star}, it is a simple matter to check that the probability mass function $\pi^*$ defined by $\pi^*(0) := \pi(0)$ and $\pi^*(i)=\sum_{j\in m(i)} \pi(j)$ for 
$i \neq 0$ satisfies the detailed balance condition and is therefore the stationary distribution for~$P^*$; indeed, using the reversibility of~$P$ with respect to~$\pi$ we have
\begin{align*}
\pi^*(0) P^*(0, i) 
&= \pi(0) \sum_{j \in m(i)} P(0, j) = \sum_{j \in m(i)} \pi(j) P(j, 0) \\ 
&= \sum_{j \in m(i)} \pi(j) (1 - \gamma_i) = \pi^*(i) P^*(i, 0).
\end{align*}

\end{document}